\documentclass[10pt]{article}
\usepackage[T1]{fontenc}
\usepackage{amssymb,amsthm,amsmath,latexsym}
\usepackage{mathrsfs}
\usepackage{physics} 
\usepackage{esint} 
\usepackage{enumitem} 
\usepackage{braket} 
\usepackage{hyperref}
\usepackage[toc,page]{appendix} 
\usepackage{graphicx,epsfig,subfigure,psfrag}
\usepackage{color}
\usepackage[usenames,dvipsnames]{pstricks}

\newtheorem{theorem}{Theorem}[section]
\newtheorem{proposition}{Proposition}[section]

\newtheorem{lemma}{Lemma}[section]
\newtheorem{corollary}{Corollary}[section]
\newtheorem{definition}{Definition}[section]
\newtheorem{remark}{Remark}[section]

\numberwithin{equation}{section} \numberwithin{theorem}{section}
\numberwithin{proposition}{section} \numberwithin{lemma}{section}
\numberwithin{corollary}{section}
\numberwithin{definition}{section} \numberwithin{remark}{section}

\newcommand{\HH}{\mathcal{H}} 

\newcommand{\R}{\mathbb{R}}

\newcommand{\diam}{{\rm diam}}

\newcommand{\com}{\color{black}}
\author{Camille Labourie and Antoine Lemenant}

\title{Regularity improvement for the minimizers of the two-dimensional Griffith energy}

\begin{document}

\maketitle

\begin{abstract}
    In this paper we prove that the singular set of connected minimizers of the
    planar Griffith functional has Hausdorff dimension strictly less then one,
    together with the higher integrability of the symetrized gradient.
\end{abstract}

\tableofcontents

\vspace{0.5cm}
\noindent {\bf Camille Labourie}, {University of Cyprus, Department of Mathematics \& Statistics, P.0. Box 20537, Nicosia, CY- 1678 Cyprus}
{\bf e-mail:} labourie.camille@ucy.ac.cy\\

\noindent{\bf Antoine Lemenant}, (corresponding author) {Universit\'e de Lorraine -- Nancy, CNRS, UMR 7502 Institut Elie Cartan de Lorraine, BP 70239
54506 Vandoeuvre-lès-Nancy}, {\bf e-mail:} {antoine.lemenant@univ-lorraine.fr}
\thispagestyle{empty}

\newpage

\section{Introduction}

In a planar elasticity setting, the Griffith energy is defined by
$$\mathcal{G}(u,K) := \int_{\Omega \setminus K} \! \mathbf A e(u):e(u) \dd{x} + \mathcal{H}^{1}(K),$$
where $\Omega \subset \R^2$ is a bounded open set which stands for the reference configuration of a linearized elastic body, and
$$\mathbf A \xi = \lambda ({\rm tr}\xi) I + 2 \mu \xi \quad \text{for all} \ \xi \in \mathbb M^{2 \times 2}_{\rm sym},$$
where $\lambda$ and $\mu$ are the Lam\'e coefficients satisfying $\mu > 0$ and $\lambda + \mu > 0$.
The energy functional is defined on pairs $(u,K)$ composed of a displacement $u : \Omega \setminus K \to \R^2$ and a $(N-1)$-dimensional crack-set $K \subset \Omega$. The notation $e(u) = (\nabla u + \nabla u^T)/2$ denotes the symmetric gradient of $u$.

The precise formulation of the Dirichlet problem is as follow.
We fix a bounded open set $\Omega'$ containing $\overline{\Omega}$ and a datum $\psi \in W^{1,\infty}(\Omega')$.
We define the admissible pairs as the elements of
$$\mathcal{A}(\Omega) := \set{K \subset {\overline \Omega} \ \text{is closed and} \ u \in LD(\Omega' \setminus K)},$$
where $LD$ is the space of functions of bounded Lebesgue deformation, i.e., functions $u \in W^{1,2}_{\mathrm{loc}}(\Omega' \setminus K)$ such that $e(u)\in L^2(\Omega' \setminus K)$.
We say that $(u,K) \in \mathcal{A}(\Omega)$ is a minimizer for the Griffith energy if it is a solution to the problem
$$\inf \Set{\int_{\Omega \setminus K} \! \mathbf A e(v):e(v) \dd{x} + \mathcal H^1(K) : \; (v,K) \in \mathcal{A}(\Omega), \, v = \psi \ \text{a.e.\ in } \ \Omega' \setminus\overline \Omega}.$$

A lot of attention has been given on the Griffith functional these last years (see  \cite{bil,CCF,CCI,CC2,CC,CFI,CFI2,FS}), and in particular it has been proved that a global minimizer $(u,K) \in \mathcal{A}(\Omega)$ (with a prescribed Dirichlet boundary condition) does exist and that the crack set $K$ is $\mathcal H^1$-rectifiable and locally Ahlfors-regular in $\Omega$.
The latter means that there exists $C_0 \geq 1$ (depending on $\mathbf{A}$) such that for all $x \in K$ and all $r > 0$ with $B(x,r) \subset \Omega$,
\begin{equation}\label{eq_AF}
    C_0^{-1} r \leq \HH^1(K \cap B(x,r)) \leq C_0 r.
\end{equation}

In \cite{bil} it was proved that any isolated connected component of the singular set $K$ of a Griffith minimizer is $C^{1,\alpha}$ a.e.\ It also applies to a connected minimizer $K$ (for e.g. minimizer with connected constraints).
In this paper we slightly improve the Hausdorff dimension of the singular set.
We also prove some higher integrability property on the symmetrized gradient.

\medskip

The main results of this paper are the following.

\medskip

\begin{theorem}
    Let $(u,K) \in \mathcal{A}(\Omega)$ be a minimizer of the Griffith energy with $K$ connected.
    Then
    \begin{enumerate}
        \item There exists $\alpha \in(0,1)$ and a relatively closed set $\Sigma \subset K \cap \Omega$ with
            $\dim_\mathcal{H}(\Sigma)<1$ such that $K \cap \Omega \setminus \Sigma$ is locally a $\mathcal C^{1,\alpha}$ curve.

        \item There exists $C \geq 1$ and $p > 1$ (depending on $\mathbf{A}$) such that for all  $x \in \Omega$ and $r > 0$ such that $B(x,r) \subset \Omega$,
            \begin{equation*}
                \int_{B(x,r/2)} \! \abs{e(u)}^{2p} \dd{x} \leq C r^{2-p}.
            \end{equation*}
    \end{enumerate}
\end{theorem}

The proof of our main theorem follows from standard technics that was already used in the scalar context of the Mumford-Shah functional, but adapted to the vectorial Griffith functional in a non trivial manner. In particular, for  (1) we follow the approach of David \cite{dMum} and Rigot \cite{rigot}, based on uniform rectifiability of the singular set and Carleson measure estimates.
The idea is to estimate to number of balls in which one can apply the $\varepsilon$-regularity theorem  contained in \cite{bil}. But the latter needs a topological separating property that one has to control in any initialized balls  which is one of the main issue of the present work (Lemma~\ref{lem_carleson1}). We also need to control the $2$-energy by a $p$-energy (Corollary \ref{corollary1}) which also uses a topological argument (Lemma~\ref{topological}). The proof of (2) is based on a strategy similar to what was first introduced by De Philippis and Figalli in \cite{DPF} and also used in \cite{LM}, which easily follows from the porosity of the singular set together with elliptic estimates. Since the needed elliptic estimates relatively to the Lam\'e system are not easy to find in the literature, we have developed an appendix containing the precise results. 

Let us stress  that the famous Cracktip function that arises as blow-up limits of   Mumford-Shah minimizers at the tip of the crack, has a vectorial analogue. This was the purpose of the work   in \cite{bcl1}. Since the vectorial Cracktip is homogeneous of degree 1/2 (see \cite[Theorem 6.4]{bcl1}), it is natural to conjecture that, akin to the standard Mumford-Shah functional, the integrability exponent of $|e(u)|$ should reach every $p<4$, as asked by De Giorgi for the Mumford-Shah functional. 

\section{Preliminaries}

\subsection*{Notation}

The $1$-dimensional Hausdorff measure is denoted by $\mathcal H^1$.
If $E$ is a measurable set, we will write $|E|$ its Lebesgue measure.
For $a$ and $b \in \R^2$, we write $a \cdot b = \sum_{i = 1}^2 a_i b_i$ the Euclidean scalar product, and we denote the norm by $|a| = \sqrt{a \cdot a}$.
The open (resp. closed) ball of center $x$ and radius $r$ is denoted by $B(x,r)$ (resp. $\overline B(x,r)$).

\medskip

We write $\mathbb M^{2 \times 2}$ for the set of real $2 \times 2$ matrices, and $\mathbb M^{2 \times 2}_{\rm sym}$ for that of all real symmetric $2 \times 2$ matrices.
Given two matrix $A,B \in \mathbb M^{2 \times 2}$, we recall the Frobenius inner product $A : B = \textrm{tr}(^t\!A B)$ and the corresponding norm $\abs{A} = \sqrt{\textrm{tr}(A^T A)}$.

\medskip

Given a weakly differentiable vector field $u$, the symmetrized gradient of $u$ is denoted by
$$e(u) := \frac{D u + D u^T}{2}.$$


\subsection*{The \texorpdfstring{$p$}{p}-normalized energy}

Let $(u,K) \in \mathcal{A}(\Omega)$.
Then for any $x_0 \in \Omega$ and $r > 0$ such that $B(x_0,r) \subset \Omega$, we define the {\it normalized elastic energy} of $u$ in $B(x_0,r)$ by
$$\omega_p(x_0,r) := r^{1-\frac{4}{p}} \left(\int_{B(x_0,r) \setminus K} | e(u)|^p \dd{x}\right)^{\frac{2}{p}}.$$

\subsection*{The flatness}
Let $K$ be a relatively closed subset of $\Omega$.
For any $x_0 \in K$ and $r > 0$ such that $B(x_0,r) \subset \Omega$, we define the {\it (bilateral) flatness} of $K$ in $B(x_0,r)$ by
$$\beta_K(x_0,r) := \frac{1}{r} \inf_{L} \max \Set{\sup_{y \in K\cap B(x_0,r)}{\rm dist }(y,L), \sup_{y \in L\cap B(x_0,r)}{\rm dist }(y,K)},$$
where $L$ belongs to the set of lines passing through $x_0$. When a minimizer $(u,K) \in \mathcal{A}(\Omega)$ is given, we write simply $\beta(x_0,r)$ for $\beta_K(x_0,r)$.

\begin{remark}
    {\rm The flatness $\beta_K(x_0,r)$ only depends on the set $K \cap B(x_0,2r)$. We have that for all $0 < t \leq r$,
        \begin{equation*}
            \beta_K(x_0,t) \leq \frac{r}{t}\beta_K(x_0,r),
        \end{equation*}
        and for $y_0 \in K \cap B(x_0,r)$ and $t > 0$ such that $B(y_0,t) \subset B(x_0,r)$,
        \begin{equation*}
            \beta_K(y_0,t) \leq \frac{2r}{t} \beta_K(x_0,r).
        \end{equation*}
    }
\end{remark}

In the sequel, we will consider the situation where $x_0 \in K$, $r > 0$ are such that $B(x_0,r) \subset \Omega$ and
\begin{equation*}
    \beta_K(x_0,r) \leq \varepsilon,
\end{equation*}
for $\varepsilon \in (0,1/2)$ small.
This implies in particular that $K\cap B(x_0,r)$ is contained in a narrow strip of thickness $\varepsilon r$ passing through the center of the ball.
Let $L(x_0,r)$ be a line passing through $x_0$ and satisfying
\begin{equation}\label{optimal}
    K \cap B(x_0,r) \subset \set{y \in B(x_0,r) | {\rm dist }(y,L)   \leq r \beta_K(x_0,r)}.
\end{equation}
We will often use a local basis (depending on $x_0$ and $r$) denoted by $(e_1,e_2)$, where $e_1$ is a tangent vector to the line $L(x_0,r)$, while $e_2$ is an orthogonal vector to $L(x_0,r)$.
The coordinates of a point $y$ in that basis will be denoted by $(y_1,y_2)$.

Provided \eqref{optimal} is satisfied with $\beta_K(x_0,r) \leq 1/2$, we can define two discs $D^+(x_0,r)$ and $D^{-}(x_0,r)$ of radius $r/4$ and such that $D^{\pm}(x_0,r) \subset B(x_0,r) \setminus K$.
Indeed, using the notation introduced above, setting $x_0^{\pm} := x_0\pm\frac{3}{4}r e_2$, we can check that $D^\pm(x_0,r) := B(x_0^\pm,r/4)$ satisfy the above requirements.
\medskip	

A property that will be fundamental in our analysis is the separation in a closed ball.

\begin{definition}\label{separation}
    Let $K$ be a relatively closed set of $\Omega$, $x_0 \in K$ and $r > 0$ be such that $B(x_0,r) \subset \Omega$ and $\beta_K(x_0,r) \leq 1/2$.
    We say that $K$ \emph{separates} $B(x_0,r)$ if the balls $D^\pm(x_0,r)$ are contained into two different connected components of $B(x_0,r) \setminus K$.
\end{definition}

The following lemma guarantees that when passing from a ball $B(x_0,r)$ to a smaller one $B(x_0,t)$, {and provided that $\beta_K(x,r)$ is relatively small,} the property of separating is preserved for $t$ varying in a range depending on {$\beta_K(x,r)$}.
\begin{lemma}\label{topological1}\cite[Lemma 3.1]{bil}
    Let $\tau \in (0,1/16)$, let $K \subset \R^2$ be a relatively closed subset of $\Omega$, let $x_0 \in K$, let $r > 0$ be such that $B(x_0,r) \subset \Omega$ and $\beta_K(x_0,r) \leq \tau$.
    If $K$ separates $B(x_0,r)$, then for all $t \in (16\tau r, r)$, we have $\beta_K(x_0,t) \leq 1/2$ and $K$ still separates $B(x_0,t)$.
\end{lemma}

\section{Local separation in many balls}

The purpose of this section is the following general result on compact connected sets which are locally Ahlfors-regular.

\begin{lemma}\label{lem_carleson1}
    Let $K \subset \overline{\Omega}$ be a compact connected set which is locally Ahlfors-regular in $\Omega$, i.e., there exists $C_0 \geq 1$ such that for all $x \in K$ and for all $r > 0$ with $B(x,r) \subset \Omega$,
    \begin{equation*}
        C_0^{-1} r \leq \HH^1(K \cap B(x,r)) \leq C_0 r.
    \end{equation*}
    Then for every $0 < \varepsilon \leq 1/2$, there exists $a \in (0,1/2)$ small enough (depending on $C_0$ and $\varepsilon$) such that the following holds. For all $x \in K$ and $r > 0$ with $B(x,r) \subset \Omega$, one can find $y \in K \cap B(x,r/2)$ and $t \in (ar,r/2)$ satisfying:
    \begin{equation}\label{eq_pair}
        \beta_K (y,t) \leq \varepsilon \ \text{and} \ K \ \text{separates} \ B(y,t) \ \text{in the sense of Definition~\ref{separation}}.
    \end{equation}
\end{lemma}

\begin{proof}
    The letter $C$ is a constant $\geq 1$ that depends on $C_0$ and whose value might increase from one line to another but a finite number of times.
    Let $x \in K$ and let $r > 0$ be such that $B(x,r) \subset \Omega$. It will be more convenient to work under the assumption that $B(x,10 r) \subset \Omega$ and $r \leq \mathrm{diam}(K) / 10$ and we are going to justify that we can make this assumption without loss of generality.
    First, we draw from the local Ahlfors-regularity that there exists a constant $C_1 \geq 1$ (depending on $C_0$) such $\mathrm{diam}(K) \geq C_1^{-1} r$.
    Let us consider the constant $\kappa := 10 C_1$, the radius $r_1 = \kappa^{-1} r$ and some $a \in (0,1/2)$.
    If we solve the problem in the ball $B(x,r_1)$, that is, if we find $y \in B(x, r_1/2)$ and $t \in (a r_1,r_1/2)$ such that (\ref{eq_pair}) holds true, then we have solved the problem in $B(x,r)$ as well because $y \in B(x,r/2)$ and $t \in (b r, r/2)$, where $b = a \kappa^{-1}$.
    This shows that it suffices to solve the problem in the ball $B(x,r_1)$ which satisfies $r_1 \leq \min(r/10, \mathrm{diam}(K)/10)$.
    From now on, we directly assume $B(x,10 r) \subset \Omega$ and $r \leq \mathrm{diam}(K)/10$.

    \vspace{0.5cm}
    \noindent\emph{Step 1. Find a smaller shifted ball with small flatness.}
    In the sequel we want to apply the results of \cite{DS4}, which works with sets of infinite diameter. This explains why we are going to need to slightly modify our set $K$ to fit in the definition of \cite{DS4}. Precisely, given an arbitrary line $L$ passing through $x$, one can check that the set
    \begin{equation*}
        E = (K \cap B(x,3r)) \cup \partial B(x,3r) \cup (L \setminus B(x,3r))
    \end{equation*}
    is connected and Ahlfors-regular in the exact sense of \cite[Definition~1.13]{DS4}, that is, $E$ is closed and there exists $C \geq 1$ (that depends on $C_0$ as usual) such that for all $y \in E$ and for all $t > 0$,
    \begin{equation*}
        C^{-1} t \leq \HH^1(E \cap B(y,t)) \leq C t.
    \end{equation*}
    As a consequence, it is contained in a (Ahlfors)-regular curve (see \cite[(1.63)]{DS4} and the discussion below) and thus is uniformly rectifiable with a constant $C \geq 1$ that depends on $C_0$ (\cite[Theorem~1.57 and Definition~1.65]{DS4}). In particular, it satisfies a geometric characterisation of uniform rectifiability called \emph{Bilateral Weak Geometric Lemma} (\cite[Definition~2.2 and Theorem~2.4]{DS4}).
    It means that for all $\varepsilon > 0$, there exists $C(\varepsilon) \geq 1$ (depending on $C_0$ and $\varepsilon$) such that for all $y \in E$ and all $t > 0$,
    \begin{equation*}
        \int_{z \in E \cap B(y,t)} \! \int_0^\rho \! \mathbf{1}_{\mathcal{C}(\varepsilon)}(z,s) \frac{\dd{s}}{s} \dd{\HH^1(z)} \leq C(\varepsilon) t,
    \end{equation*}
    where
    \begin{equation*}
        \mathcal{C}(\varepsilon) := \set{(z,s) | z \in E, \ s > 0, \ \text{and} \ \beta_E(z,s) > \varepsilon}.
    \end{equation*}
    We apply this property with $y := x$ and $\rho := r$,
    \begin{equation}\label{eq_carlesonE}
        \int_{z \in E \cap B(x,r)} \! \int_0^\rho \! \mathbf{1}_{\mathcal{C}(\varepsilon)}(z,s) \frac{\dd{s}}{s} \dd{\HH^1(z)} \leq C(\varepsilon) r,
    \end{equation}
    We observe that for all $z \in K \cap B(x,r)$ and for all $0 < s < r$, we have $B(z,2s) \subset B(x,3r)$ so $K \cap B(z,2s) = E \cap B(z,2s)$ and hence $\beta_K(z,s) = \beta_E(z,s)$. Therefore, (\ref{eq_carlesonE}) simplifies to
    \begin{equation}\label{eq_carleson1}
        \int_{z \in K \cap B(x,r)} \! \int_{0}^r \! \mathbf{1}_{\mathcal{B}(\varepsilon)}(z,s) \frac{\dd{s}}{s} \dd{\HH^1(z)} \leq C(\varepsilon) r,
    \end{equation}
    where
    \begin{equation*}
        \mathcal{B}(\varepsilon) := \Set{(z,s) | z \in K, \ s > 0, \ B(z,s) \subset \Omega \text{and} \ \beta_K(z,s) > \varepsilon}.
    \end{equation*}

    We fix $\varepsilon \in (0,1/10)$.
    We let $a \in (0,1/2)$ be a small parameters that will be fixed later.
    Assume by contradiction that for all $z \in K \cap B(x,r)$ and $s \in (ar,r)$, we have $\beta_K(z,s) > \varepsilon$.
    This means that for such pairs $(z,s)$, we have ${\bf 1}_{\mathcal{B}(\varepsilon)}(z,s) = 1$.
    Moreover we have by local Ahlfors-regularity $\HH^1(K \cap B(x,r)) \geq C_0^{-1} r$ so
    \begin{align*}
        \int_{z \in K\cap B(x,r)} \! \int_{0}^r \! {\bf 1}_{\mathcal{B}(\varepsilon)}(z,s) \frac{\dd{s}}{s} \dd{\HH^1(z)} &\geq \mathcal{H}^1(K\cap B(x,r)) \int_{ar}^r \! \frac{\dd{s}}{s}\\
                                                                                                                            &\geq \mathcal{H}^1(K\cap B(x,r)) \ln(\frac{1}{a})\\
                                                                                                                            &\geq C_0^{-1} r\ln(\frac{1}{a}).
    \end{align*}
    Using now \eqref{eq_carleson1}, we arrive at a contradiction, provided that $a$ is small enough (depending on $C_0$ and $\varepsilon$).
    Thus we have found $y \in B(x,r/2)$ and $t \in (ar ,r/2)$ such that $t \leq \mathrm{diam}(K)/10$ and
    $$\beta_K(y,t) \leq \varepsilon.$$

    \vspace{0.5cm}
\noindent\emph{Step 2. Conclusion.}
    We then need to find a shifted smaller ball again in which $K$ separates. For that purpose, we will use the fact that a compact connected set with finite length is arcwise connected (see \cite[Theorem 1.8]{DS4}).
    We choose a coordinate system such that $y = (0,0)$ and the line $L$ that realizes the infimum in the definition of $\beta_K(y,t)$ is the $x$ axis: $L = \R\times \set{0}$, in particular $\beta_K(0,t) \leq \varepsilon$ and
    \begin{equation*}
        K\cap B(0,t) \subset \set{(z_1,z_2) | \abs{z_2} \leq \varepsilon t.}
    \end{equation*}
    Since $t \leq \diam(K)/10$, there exists a point $z \in K \setminus B(0,t)$ and a curve $\Gamma \subset K$ from $0$ to $z$.
    This curve touches $\partial B(0,t)$ and we let $z' \in \partial B(0,t)$ be the point of $\partial B(0,t)$ which is reached by $\Gamma$ for the first time.
    Let $\Gamma' \subset \Gamma$ be the piece of curve from $0$ to $z'$. We see that $$\Gamma' \setminus \set{z} \subset K \cap B(0,t)$$ since $\Gamma'$ starts at $0$ and meets $\partial B(0,t)$ only at its extremity. Therefore, we have $$\Gamma' \subset \overline{B}(0,t) \cap \set{(z_1,z_2) | \abs{z_2} \leq \varepsilon t}.$$
    The point $z'$ must lie either on the arc $\partial B(0,t)\cap \set{(z_1,z_1) | z_1 > 0}$ or on the arc $\partial B(0,t)\cap \set{(z_1,z_1) | z_1 < 0}$.
    Let us assume that the first case occurs (for the second case we can argue similarily).
    The curve $\Gamma'$ stays inside $\overline{B}(0,t) \cap \set{(z_1,z_2) | \abs{z_2} \leq \varepsilon t}$, and runs from $0$ (the center of the ball) to $z'$ (on the boundary of $B(0,t)$) such that $z'_1 > 0$.
    We can therefore find a point $y' = (y'_1,y'_2) \in \Gamma'$ such that $y'_1 > 0$, $\abs{y'_2} \leq \varepsilon t$ and $t/4 \leq \abs{y'} \leq 3t/4$. 
    Moreover, $\Gamma'$ must separate $B(y',t/8)$ the two connected components of
    \begin{equation}\label{eq_strip0}
         B(y',t/8) \cap \set{(z_1,z_2) | \abs{z_2} > \varepsilon t}.
    \end{equation}
    and thus it separates $B(y',t/8)$ the two connected components of
    \begin{equation}\label{eq_strip}
         B(y',t/8) \cap \set{(z_1,z_2) | \abs{z_2 - y'_2} > 2\varepsilon t}.
    \end{equation}
    The strip in (\ref{eq_strip}) contains the strip in (\ref{eq_strip0}) so that it is centred on $y'$.
    As $K$ contains $\Gamma'$, then $K$ also separates $B(y',t/8)$.
    Moreover, we have $\beta_K(y',t/8) \leq 16 \beta_K(y,t) \leq 16 \varepsilon$ and one can replace $\varepsilon$ by $\varepsilon/16$ in the proof above to arrive exactly at $\beta_K(y',t/8) \leq \varepsilon$.
\end{proof}

\section{Carleson measure estimates on \texorpdfstring{$\omega_p$}{omega\_p}}

We define
$$\Delta := \set{(x,r) | x \in K,\ r > 0 \ \text{and} \ B(x,r) \subset \Omega}.$$
The purpose of this section is to state the following fact.

\begin{proposition}\label{carleson2}
    Let $(u,K) \in \mathcal{A}(\Omega)$ be a minimizer of the Griffith functional.
    For all $p \in [1,2)$, there exists $C_p \geq 1$ (depending on $p$ and $\bf A$) such that
    $$\int_{y \in K\cap B(x,r)} \! \int_{0<t<r} \! \omega_p(y,t) \frac{\dd{t}}{t} \dd{\HH^1(y)} \leq C_p r,$$
    for all $(x,2r) \in \Delta$.
\end{proposition}

\begin{proof}
    The proof was originally performed by David and Semmes in the scalar context of Mumford-Shah minimizers (see \cite[Section 23]{d}).
    It relies on the local Ahlfors-regulary of Griffith minimizers, that is, there exists $C_0 \geq 1$ (depending on $\bf A$) such that for all $(x,r) \in \Delta$,
    \begin{equation}\label{eq_AF2}
        \int_{B(x,r)} \! |e(u)|^2 \dd{x} + \mathcal{H}^1(K\cap B(x,r)) \leq C_0 r
    \end{equation}
    and
    \begin{equation*}
        \mathcal{H}^1(K\cap B(x,r)) \geq C_0^{-1} r.
    \end{equation*}
    The inequality (\ref{eq_AF2}) directly follows by taking
    \begin{equation*}
        u{\bf 1}_{\Omega \setminus B(x,r)} \quad \text{and} \quad (K \setminus B(x,r) )\cup \partial B(x,r)
    \end{equation*}
    as a competitor, and the ellipticity of $\mathbf A$.
    The proof in \cite[Section 23]{d} on Mumford-Shah minimizers can be followed verbatim so we prefer to omit the details and refer directly to \cite{d}.
\end{proof}


\section{Control of \texorpdfstring{$\omega_2$}{omega\_2} by \texorpdfstring{$\omega_p$}{omega\_p}}

The main $\varepsilon$-regularity theorem uses an assumption on the smallness of $\omega_2$.
Unfortunately, what we can really control in many balls (thanks to Proposition \ref{carleson2}) is $\omega_p$ for $p<2$, which is weaker.
This is why in this section we prove that $\omega_2$ can be estimated from $\omega_p$, for a minimizer.
This strategy was already used in \cite{d} and \cite{rigot} for the Mumford-Shah functional.
The adaptation for the Griffith energy is not straigthforward, but can be done by following a similar approach as the one already used in \cite[Section 4.1.]{bil}, generalized with $\omega_p$ instead of only $\omega_2$.
Some estimates from the book \cite{d} were also useful.

\begin{lemma}[Harmonic extension in a ball from an arc of circle]\label{extensionL}
    Let $p \in (1,2]$, $0 < \delta \leq 1/2$, $x_0 \in \mathbb R^2$, $r > 0$ and
    let $\mathscr C_\delta \subset\partial B(x_0,r)$
    be the arc of circle defined by $$\mathscr C_\delta := \set{(x_1,x_2) \in \partial B(x_0,r) \; : \; (x-x_0) \cdot {\bf e}_2 > \delta r}.$$
    Then, there exists a constant $C \geq 1$ (depending on $p$) such that every function $u \in W^{1,p}(\mathscr C_\delta;\R^2)$ extends to a function $g \in W^{1,2}(B(x_0,r);\R^2)$ with $g = u$ on $\mathscr C_\delta$ and
    $$\int_{B(x_0,r)} \! |\nabla g|^2 \dd{x} \leq C r^{2-\frac{2}{p}} \left(\int_{\mathscr C_\delta} \! |\partial_\tau u|^p \, d\mathcal{H}^1\right)^{\frac{2}{p}}.$$
\end{lemma}

\begin{proof}
    Let $\Phi : \mathscr C_\delta\to \mathscr C_0$ be a bijective and bilipschitz mapping from $\mathscr C_\delta$ to $\mathscr C_0$. As $\delta \leq 1/2$, we can can request that the biLipschitz constant of $\Phi$ is universal.
    Since $u\circ \Phi^{-1} \in W^{1,p}(\mathscr C_0;\R^2)$, we can define the extension by reflection $\tilde u \in W^{1,p}(\partial B(x_0,r); \R^2)$ on the whole circle $\partial B(x_0,r)$, that satisfies
    $$\int_{\partial B(x_0,r)}\! |\partial_\tau \tilde u|^p d\mathcal{H}^1 \leq C \int_{\mathscr C_\delta} \! |\partial_\tau u|^p\, d\mathcal{H}^1,$$
    where $C \geq 1$ is a universal constant.
    We next define $g$ as the harmonic extension of $\tilde u$ in $B(x_0,r)$. Using \cite[Lemma 22.16]{d}, we obtain
    $$\int_{B(x_0,r)} \! |\nabla g|^2 \dd{x} \leq Cr^{2-\frac{2}{p}} \left(\int_{\partial B(x_0,r)}\! |\partial_\tau \tilde u|^p\, d\mathcal{H}^1 \right)^{\frac{2}{p}} \leq C r^{2-\frac{2}{p}} \left(\int_{\mathscr C_\delta}\! |\partial_\tau u|^2 \, d\mathcal{H}^1\right)^{\frac{2}{p}},$$
    which completes the proof.
\end{proof}

\begin{lemma}\label{extension}
    Let $p \in (1/2]$, let $(u,K) \in \mathcal A(\Omega)$ be a minimizer of the Griffith functional. Let $x_0 \in K$ and $r > 0$ be such that $B(x_0,r) \subset \Omega$ and $\beta(x_0,r) \leq 1/2$.
    Let $S$ be the strip defined by
    $$S := \set{y \in B(x_0,r) | {\rm dist}(y,L) \leq r\beta(x_0,r)},$$
    where $L$ is the line passing through $x_0$ which achieves the infimum in $\beta_K(x_0,r)$.
    Then there exists a constant $C \geq 1$ (depending on $p$), a radius $\rho \in (r/2,r)$ and two functions $v^\pm \in W^{1,2}(B(x_0,\rho);\R^2)$ such that $v^\pm = u \ \text{on} \ \mathscr C^\pm$ and
    $$\int_{B(x_0,\rho)}\! |e( v^\pm)|^2 \dd{x} \leq Cr^{2-\frac{4}{p}} \left(\int_{B(x_0,r) \! \setminus K}|e(u)|^p \dd{x}\right)^{\frac{2}{p}},$$
    where $\mathscr C^\pm$ are the connected components of $\partial B(x_0,\rho) \setminus S$
\end{lemma}

\begin{proof}
    Let $A^\pm$ be the connected components of $B(x_0,r) \setminus S$.
    Since $K\cap A^\pm = \emptyset$, by Korn inequality there exists two skew-symmetric matrices $R^\pm$ such that the functions $x\mapsto u(x)-R^\pm x$ belong to $W^{1,p}(A^\pm;\R^2)$ and
    $$\int_{A^\pm} \! |\nabla u-R^\pm|^p \dd{x} \leq C \int_{A^\pm} \! |e(u)|^p \dd{x},$$
    where the constant $C \geq 1$ is universal since the domains $A^\pm$ are all uniformly Lipschitz for all possible values of $\beta(x_0,r) \leq 1/2$.
    Using the change of variables in polar coordinates, we infer that
    $$\int_{A^\pm} \! |\nabla u-R^\pm|^p \dd{x} = \int_{0}^{r} \! \left(\int_{\partial B(x_0,\rho)\cap A^\pm} \! |\nabla u-R^\pm|^p \, d\mathcal{H}^1\right) d\rho$$
    which allows us to choose a radius $\rho \in (r/2,r)$ satisfying
    \begin{multline*}
        \int_{\partial B(x_0,\rho)\cap A^+} \! |\nabla u-R^+|^p\,d\mathcal{H}^1 + \int_{\partial B(x_0,\rho)\cap A^-} \! |\nabla u-R^-|^p\,d\mathcal{H}^1\\
        \leq \frac{2}{r} \int_{A^+} \! |\nabla u-R^+|^p \dd{x} + \frac{2}{r} \int_{A^-} \! |\nabla u-R^-|^p \dd{x} \leq \frac{C}{r} \int_{B(x_0,r) \setminus K} \! |e(u)|^p \dd{x}.
    \end{multline*}
    Setting $\mathscr C^\pm := \partial A^\pm \cap \partial B(x_0,\rho)$, in view of Lemma~\ref {extensionL} applied to the functions $u^\pm:x \mapsto u(x)-R^\pm x$, which belong to
    $W^{1,p}(\mathscr C^\pm;\R^2)$ since they are regular, for {\com $\delta = \beta(x_0,r)$} we get two functions $g^\pm \in W^{1,2}(B(x_0,\rho);\R^2)$ satisfying $g^\pm(x) = u(x)-R^\pm x$ for $\mathcal H^1$-a.e.\ $x \in \mathscr C^\pm$ and
    \begin{align*}
        \int_{B(x_0,\rho)}\! |\nabla g^\pm|^2 \dd{x} &\leq C\rho^{2-\frac{2}{p}} \left(\int_{\mathscr C^\pm}\! |\partial_\tau u^\pm|^p \,d\mathcal H^1\right)^{\frac{2}{p}}\\
                                                     &\leq Cr^{2-\frac{4}{p}} \left(\int_{B(x_0,r) \setminus K}\! |e(u)|^p \dd{x}\right)^{\frac{2}{p}}.
    \end{align*}
    Finally, the functions $x \mapsto v^\pm(x) := g^\pm(x) + R^\pm x$ satisfy the required properties.
\end{proof}

Using the competitor above, we can obtain the following.

\begin{proposition}\label{prop1}
    Let $p \in (1,2]$, let $(u,K) \in \mathcal A(\Omega)$ be a minimizer of the Griffith functional. Let $x_0 \in K$ and $r > 0$ be such that $B(x_0,r) \subset \Omega$ and $\beta(x_0,r) \leq 1/2$.
    Then there exists a constant $C \geq 1$ (depending on $p$) and a radius $\rho \in (r/2,r)$ such that
    $$\int_{B(x_0,\rho) \setminus K} \! \mathbf Ae(u):e(u) \dd{x} + \mathcal{H}^1(K\cap B(x_0,\rho)) \leq 2\rho + C\rho\big(\omega_p(x_0,r) + \beta(x_0,r)\big).$$
\end{proposition}

\begin{proof}
    We keep using the same notation than that used in the proof of Lemma~\ref {extension}.
    Let $\rho \in (r/2,r)$ and $v^{\pm} \in W^{1,2}(B(x_0,\rho);\R^2)$ be given by Lemma~\ref {extension}.
    We now construct a competitor in $B(x_0,\rho)$ as follows.
    First, we consider a ``wall'' set $Z \subset \partial B(x_0,\rho)$ defined by
    $$Z := \set{y \in \partial B(x_0,\rho) | {\rm dist}(y,L(x_0,r)) \leq r \beta(x_0,r)}.$$
    Note that $K\cap \partial B(x_0,\rho) \subset Z$,
    $$\partial B(x_0,\rho) = [\partial A^+\cap \partial B(x_0,\rho)]\cup [\partial A^-\cap \partial B(x_0,\rho)] \cup Z = \mathscr C^+\cup \mathscr C^- \cup Z,$$
    and that
    $$\mathcal{H}^1(Z) = 4\rho \arcsin\left( \frac{r\beta(x_0,r)}{\rho} \right) \leq 2 \pi r \beta(x_0,r).$$

    We are now ready to define the competitor $(v,K')$ by setting
    $$K' := \big[K \setminus B(x_0,\rho ) \big]\cup Z \cup \big[L(x_0,r)\cap B(x_0,\rho)\big],$$
    and, denoting by $V^\pm$ the connected components of $B(x_0,\rho) \setminus L(x_0,r)$ which intersect $A^\pm$,
    $$
    v :=
    \begin{cases}
        {v^\pm} & \text{in} \ V^\pm \\
        u       & \text{otherwise.}
    \end{cases}
    $$
    Since $\mathcal{H}^1(K'\cap \overline B(x_0,\rho)) \leq 2\rho + 2 \pi r \beta(x_0,r)$, we deduce that
    \begin{align*}
        \begin{split}
                     & \int_{B(x_0,\rho) \setminus K} \! \mathbf Ae(u):e(u) \dd{x} +\mathcal{H}^1(K \cap \overline B(x_0,\rho))                   \\
                     & \leq \int_{B(x_0,\rho) \setminus K} \! \mathbf Ae(v):e(v) \dd{x} + \mathcal{H}^1(K' \cap \overline B(x_0,\rho))
        \end{split} \\
                     & \leq Cr^{2-\frac{4}{p}} \left(\int_{B(x_0,r) \setminus K} \! |e(u)|^p \dd{x}\right)^{\frac{2}{p}} + \rho(2 + C\beta(x_0,r)) \\
                     & \leq 2\rho + C \rho\big(\omega_p(x_0,r) + \beta(x_0,r)\big),
    \end{align*}
    and the proposition follows.
\end{proof}

The next Lemma is of purely topological nature.

\begin{lemma}\label{topological}
    Let $K \subset \R^2$ be a compact connected set with $\mathcal{H}^1(K) < +\infty$.
    For all $x \in K$ and $r \in (0,\diam(K)/2)$, we have
    $$\mathcal{H}^1(K\cap B(x,r)) \geq 2r - 4r\beta_K(x,r).$$
\end{lemma}

\begin{proof}
    The inequality is trivial if $\beta_K(x,r) \geq 1/4$ so we can assume $\beta_K(x,r) \leq 1/4$ without loss of generality.
    The difficulty here is that even though $K$ is connected, it may be that $K\cap B(x,r)$ is not.
    Let $\varepsilon$ be such that $\beta_K(x,r) < \varepsilon \leq 1/2$.
    To simplify the notations, we assume that $x = (0,0)$ and that horizontal axis $L = \R \times \set{0}$ achieves the infimum in the definition of $\beta_K(x,r)$. Therefore, $K\cap B(x,r)$ is contained in the strip $S$ defined by
    $$S := B(x,r) \cap \set{(z_1,z_2) \in \R^2 | |z_2| \leq \varepsilon r}.$$
    Let $\pi_1:\R^2\to L$ be the orthogonal projection onto $L$.
    As $\pi_1$ is $1$-Lipschitz, we know that
    $$\mathcal{H}^1(K\cap B(x,r)) \geq \mathcal{H}^1(\pi_1(K\cap B(x,r))).$$
    Let us define $E := \pi_1(K\cap B(x,r))$ and introduce the constant
    $$c_\varepsilon := \sqrt{1-\varepsilon^2}.$$
    We are going to use the fact that $K$ is connected to show that $L\cap [-r c_\varepsilon , r c_\varepsilon] \setminus E$ is an interval (we identify $L$ with the real axis).
    For each $a \in E$, there exists $|t| \leq \varepsilon r$ such that $z_0 := (a,t) \in K$, and since $r < \diam(K)/2$ there exists a curve $\Gamma$ that connects $z_0$ to some point $z_1 \in K \setminus B(x,r)$.
    But then $E$ has to contain $\pi_1(\Gamma)$, and since $\beta_K(x,r) \leq 1/4$ it means that either $[a,rc_\varepsilon] \subset E$ or $[-rc_\varepsilon,a] \subset E$.
    Indeed, the curve $\Gamma$ is contained in the strip $S$ and has to ``escape the ball'' $B(x,r)$ either from the right or from the left.
    The projection with minimal length would be when $\Gamma$ escapes exactly at the corner of $S\cap B(x,r)$ which gives the definition of $c_\varepsilon$ (see the picture below).

    This holds true for all $a \in E$, which necessarily imply that $[- c_\varepsilon r, c_\varepsilon r] \setminus E$ is an interval, that we denote by $I$. As $\left(I \times [-\varepsilon r, \varepsilon r]\right) \cap K = \emptyset$, we must have $\abs{I} \leq 2 \varepsilon r$ otherwise the center of $I$ is a distance $\geq \varepsilon r$ from $K$ whence $\beta_K(x,r) \geq \varepsilon$, which would contredict the definition of $\varepsilon$.
    All in all we have proved that
    $$\mathcal{H}^1(K\cap B(x,r)) \geq \mathcal{H}^1(\pi_1(K\cap B(x,r))) \geq 2rc_\varepsilon - 2 \varepsilon r.$$
    Now, we use $\sqrt{1 - \varepsilon^2} \geq 1 - \varepsilon$ to estimate
    \begin{align*}
        2 c_\varepsilon r - 2 \varepsilon r &= 2 r \left(\sqrt{1 - \varepsilon^2} - \varepsilon\right)\\
                                            &\geq 2 r \left(1 - 2 \varepsilon\right).
    \end{align*}
    Since $\varepsilon$ can be chosen arbitrary close to $\beta_K(x,r)$, the result follows.
    \begin{center}
        \begin{figure}[ht]
            \psscalebox{0.8 0.8} 
            {
                \begin{pspicture}(0,-5.653333)(12.935169,3.0666666)
                    \definecolor{colour1}{rgb}{0.2,0.3019608,0.7019608}
                    \definecolor{colour0}{rgb}{0.3019608,0.3019608,0.3019608}
                    \psellipse[linecolor=black, linewidth=0.04, dimen=outer](6.08,-1.2933333)(4.5933332,4.36)
                    \psdots[linecolor=black, dotsize=0.17653137](6.0533333,-1.2666667)
                    \rput[bl](8.66,2.68){$B(x,r)$}
                    \rput[bl](5.886667,-0.6){$x$}
                    \psline[linecolor=black, linewidth=0.04](10.42,0.093333334)(1.7266667,0.04)
                    \psline[linecolor=black, linewidth=0.04](1.6733333,-2.5333333)(10.513333,-2.4933333)(10.513333,-2.4933333)
                    \psline[linecolor=black, linewidth=0.04](0.08666665,-1.3066666)(12.46,-1.2666667)
                    \rput[bl](11.953333,-1.1133333){$L$}
                    \psbezier[linecolor=black, linewidth=0.04](7.7266665,-0.44)(8.131787,-0.95456684)(7.086648,-1.2861348)(6.0866666,-1.28)(5.0866857,-1.2738651)(3.8177464,-1.1199914)(2.9133334,-0.6933333)(2.0089204,-0.26667523)(2.1646423,-1.0502632)(1.3666667,0.94666666)(0.56869096,2.9435966)(-0.3955932,0.68546563)(0.56666666,0.41333333)
                    \psbezier[linecolor=black, linewidth=0.04](9.713333,-0.6666667)(8.752987,-0.9454769)(10.284147,-2.0273814)(11.206667,-2.4133333333333304)(12.129186,-2.7992852)(11.101433,-1.2380844)(11.953333,-1.8266667)
                    \rput[bl](11.873333,-2.4){$K$}
                    \psline[linecolor=black, linewidth=0.04, linestyle=dashed, dash=0.17638889cm 0.10583334cm](7.806667,0.44)(7.8333335,-2.88)
                    \psline[linecolor=black, linewidth=0.04, linestyle=dashed, dash=0.17638889cm 0.10583334cm](9.3133335,0.6)(9.42,-2.9733334)
                    \rput[bl](8.333333,-1.1333333){$I$}
                    \psline[linecolor=black, linewidth=0.04, linestyle=dashed, dash=0.17638889cm 0.10583334cm](1.7533333,0.04)(1.74,-2.48)
                    \psline[linecolor=black, linewidth=0.04, arrowsize=0.05291667cm 2.0,arrowlength=1.4,arrowinset=0.0]{<-}(1.6066667,-1.2)(1.1266667,-0.38666666)
                    \psline[linecolor=colour1, linewidth=0.08, tbarsize=0.07055555cm 5.0]{|*-|*}(7.82,-1.2666667)(9.353333,-1.2666667)
                    \psline[linecolor=colour0, linewidth=0.04, tbarsize=0.07055555cm 5.0]{|*-|*}(11.02,0.10666667)(11.033334,-1.24)
                    \rput[bl](11.353333,-0.5466667){$\varepsilon r$}
                    \rput[bl](0.0,-0.23333333){$\sqrt{1 - \varepsilon^2} r$}
                \end{pspicture}
            }
        \end{figure}

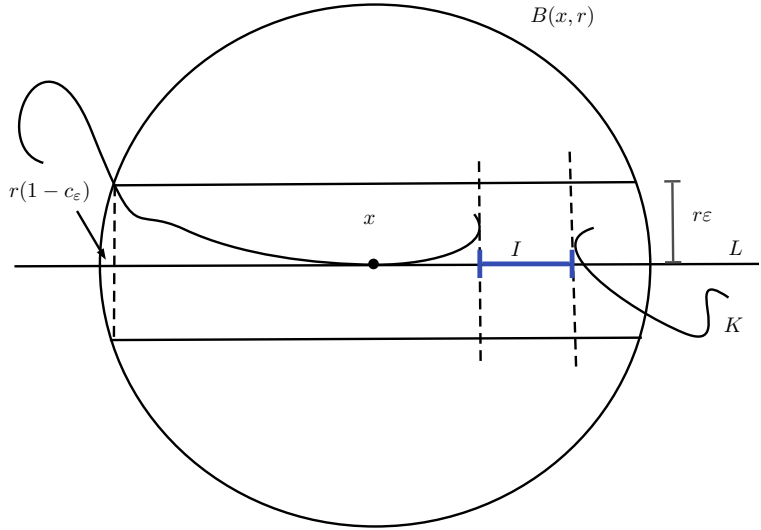
        Figure 1: estimating the length of $K\cap B(x,r)$.
    \end{center}
\end{proof}

We now come to the interesting ``reverse H\"older'' type estimate that will be needed later.

\begin{corollary}\label{corollary1}
    Let $p \in (1,2]$, let $(u,K) \in \mathcal A(\Omega)$ be a minimizer of the Griffith functional. Let $x_0 \in K$ and $r \in (0,\diam(K))$ be such that $B(x_0,r) \subset \Omega$. Then there exists a constant $C_p \geq 1$ (depending on $\bf A$ and $p$) and a radius $\rho \in (r/2,r)$ such that
    $$\omega_2(x_0,\rho) \leq C_p \big(\omega_p(x_0,r) + \beta(x_0,r)\big).$$
\end{corollary}

\begin{proof}
    If $\beta(x_0,r) \geq 1/2$, the inequality is straightforward. Indeed we know by the usual upper bound (\ref{eq_AF2}) that we have $\omega_2(x_0,\rho) \leq C$ for some constant $C \geq 1$. Therefore in this case, it suffices to choose $C_p \geq 2C$ in Corollary \ref{corollary1}.
    We now assume that $\beta(x_0,r) \leq 1/2$.
    By Proposition~\ref{prop1}, we already know that there exists a constant $C \geq 1$ and a radius $\rho \in (r/2,r)$ such that
    $$\int_{B(x_0,\rho) \setminus K} \! \mathbf Ae(u):e(u) \dd{x} + \mathcal{H}^1(K\cap B(x_0,\rho)) \leq 2\rho + C\rho\big(\omega_p(x_0,r) + \beta(x_0,r)\big).$$
    Then, we apply Lemma~\ref{topological} in $B(x_0,\rho)$ which yields
    $$\mathcal{H}^1(K\cap B(x_0,\rho)) \geq 2\rho - 4 \rho \beta(x_0,r\rho)$$
    hence,
    $$\int_{B(x_0,\rho) \setminus K} \! \mathbf Ae(u):e(u) \dd{x} \leq C\rho\big(\omega_p(x_0,r) + \beta(x_0,r)\big).$$
    Finally, by ellipticity of $A$ we get
    $$\int_{B(x_0,\rho) \setminus K} \! \mathbf Ae(u):e(u) \dd{x} \geq C^{-1} \int_{B(x_0,\rho)} \abs{e(u)}^2 \dd{x} = C^{-1} \omega_2(x_0,\rho) \rho,$$
    which finishes the proof.
\end{proof}


\section{Porosity of the bad set}

Given $0 < \alpha < 1$, $x_0 \in K$ and $r > 0$ such that $B(x,r) \subset \Omega$, we say that the crack-set $K$ is $C^{1,\alpha}$-regular in the ball $B(x_0,r)$ if it is the graph of a $C^{1,\alpha}$ function $f$ such that, in a convenient coordinate system it holds $f(0) = 0$, $f'(0) = 0$ and $r^\alpha \|f'\|_{C^\alpha} \leq 1/16$.
We recall the following $\varepsilon$-regularity theorem coming from \cite{bil}. 

\begin{theorem}\cite{bil}\label{BILth}
    Let $(u,K) \in \mathcal A(\Omega)$ be a minimizer of the Griffith functional with $K$ connected.
    There exists constants $a,\alpha ,\varepsilon_2 \in (0,1)$ (depending on $\bf A$) such that the following property holds true.
    Let $x_0 \in K$ and $r > 0$ be such that $B(x_0,r) \subset \Omega$ and
    $$\omega_2(x_0,r) + \beta(x_0,r) \leq \varepsilon_2,$$
    and $K$ separates $B(x_0,r)$.
    Then $K$ is $C^{1,\alpha}$-regular in $B(x_0,ar)$.
\end{theorem}

\begin{proof}
    Unfortunately, the above statement is not explicitely stated in \cite{bil}, but it directly follows from the proof of \cite[Proposition~3.4]{bil}.
    Indeed, in the latter proof, some explicit thresholds $\delta_1 > 0$ and $\delta_2 > 0$ and an exponent $\alpha \in (0,1)$ are given so that, provided that
    $$\omega_2(x_0,r) \leq \delta_2, \quad \beta(x_0,r) \leq \delta_1, $$
    and $K$ separates $B(x_0,r)$, then
    $$\beta(y,t) \leq C \left(\frac{t}{r}\right)^\alpha$$
    for all $y \in B(x_0,r/2)$ and $t \in (0,r/2)$.
    It implies that there exists $a \in (0,1)$ (which depends on $C$ and $\alpha$) such that $B(x_0,a r)$ is a $C^{1,\alpha}$ and $10^{-2}$-Lipschitz graph above a linear line (thanks to \cite[Lemma 6.4]{bil}). The line can be chosen to be the tangent line to $K$ at $x_0$ so that the graph satisfies $f'(0) = 0$.
    In addition the estimate (6.8) in \cite{bil} says moreover $r^{\alpha} \|f'\|_{C^{0,\alpha}}\leq C$, from which we easily get $(a r)^\alpha\|f'\|_{C^{0,\alpha}}\leq 1/16$ up to take a smaller constant $a$.
    The fact that $\alpha$ and $a$ depend only on $\bf A$, follows from a careful inspection of the proof in \cite{bil}.
\end{proof}

We are now in position to prove the following, which says that the singular set is porous in $K$.

\begin{proposition}\label{porosity}
    Let $(u,K) \in \mathcal A(\Omega)$ be a minimizer of the Griffith functional with $K$ connected.
    There exists constants $a \in (0,1/2)$ (depending on $\bf A$) such that the following property holds true.
    For all $x_0 \in K$ and $r > 0$ such that $B(x_0,r) \subset \Omega$, there exists $y \in K\cap B(x_0,r/2)$ such that $K\cap B(y,ar)$ is $C^{1,\alpha}$-regular (where $\alpha$ is the constant of Theorem~\ref{BILth}).
\end{proposition}

\begin{proof}
    In view Theorem~\ref{BILth}, it is enough to prove the following fact: there exists $a \in (0,1/2)$ such that for all $x \in K$ and $r > 0$ with $B(x,r) \subset \Omega$, there exists $y \in K\cap B(x,r/2)$ and $ar < s < r/2$ such that
    $$\omega_2(y,s) + \beta(y,s) \leq \varepsilon_2$$
    and $K$ separates $B(y,s)$, where $\varepsilon_2$ is the constant of Theorem~\ref{BILth}.
    We already know from Lemma~\ref{lem_carleson1} how to control $\beta$ and the separation.
    We therefore need to add a control on $\omega_2$, and this will be done by applying successively Proposition~\ref{carleson2} and Corollary \ref{corollary1}, but we need to fix carefully the constants so that it compiles well.

    Let us pick any $p \in (1,2)$ and let $C_p$ be a constant which is bigger than the constant of Proposition~\ref{carleson2} and the constant of Corollary \ref{corollary1}.
    Let also $C_0$ be the constant used for Ahlfors-regularity, see (\ref{eq_AF}).
    Then we define
    $$b := \frac{1}{2} \exp\left(\frac{-16 C_0 C_p^2}{\varepsilon_2}\right)$$
    and we let $\varepsilon_0 \in (0,1/2)$ be a small constant which will be fixed during the proof and will depend only on $\mathbf{A}$.
    We fix $x \in K$ and $r > 0$ such that $B(x,r) \subset \Omega$.
    As noticed at the beginning of the proof of Lemma~\ref{lem_carleson1}, we can assume without loss of generality that $B(x,10 r) \subset \Omega$ and $r \leq \mathrm{diam}(K) / 10$.
    We apply Lemma~\ref{lem_carleson1} with the previous definition of $\varepsilon_0$ and we get that for some $a \in (0,1/2)$ (depending on $\mathbf{A}$), there exists $y \in B(x,r/2)$ and $t \in (ar,r/2)$ satisfying:
    $$\beta (y,t) \leq \varepsilon_0 \ \text{and} \ K \ \text{separates} \ B(y,t) \ \text{as in Definition~\ref{separation}}.$$ Note that the assumption of Lemma~\ref{lem_carleson1} is satisfied thanks to (\ref{eq_AF}).
    Then we apply Proposition~\ref{carleson2} in $B(y,t)$ which yields
    \begin{eqnarray}
        \int_{z \in K\cap B(y,t)} \! \int_{0<s<t} \! \omega_p(z,s) \frac{\dd{s}}{s} \dd{\HH^1(z)} \leq C_p t.\label{intint}
    \end{eqnarray}
    From this estimate we claim that we obtain the following fact:
    \begin{eqnarray}
        \text{there exists} \ z \in B(y,t/2) \ \text{and} \ s \in (b t, t/2) \ \text{such that} \ \omega_p(z,s) \leq \frac{\varepsilon_2}{4C_p}.\label{claim}
    \end{eqnarray}
    Indeed, remember that by Ahlfors-regularity $\HH^1(K \cap B(y,t/2)) \geq C_0^{-1} (t/2)$, see (\ref{eq_AF}), and if the claim (\ref{claim}) is not true, then
    \begin{align*}
        \int_{z \in K\cap B(y,t/2)} \! \int_{b t < s < t/2} \! \omega_p(z,t) \; \frac{\dd{s}}{s} \dd{\HH^1(z)} &\geq \frac{\varepsilon_2}{4C_p} \mathcal{H}^1(K\cap B(y,t/2)) \int_{b t}^{t/2} \! \frac{\dd{t}}{t}\\
                                                                                                              &\geq \frac{t\varepsilon_2}{8C_p C_0} \ln(\frac{1}{2 b}).
    \end{align*}
    Returning back to \eqref{intint}, we get
    $$\ln(\frac{1}{2 b}) \leq \frac{8C_0 C_p^2}{\varepsilon_2}$$
    which contredicts our definition of $b$. The claim is now proved.

    Now let us check what we have got in the ball $B(z,s)$.
    We already know that $s \geq c r$ for some constant $c \in (0,1/2)$ (which depends on on $\mathbf{A}$) and $\omega_p(z, s) \leq \frac{\varepsilon_2}{4C_0}$. Concerning $\beta$, we have
    $$\beta(z,s) \leq \frac{2}{b}\beta(y,t) \leq \frac{2}{b}\varepsilon_0$$
    and here we can assume $\varepsilon_0$ small enough such that $\beta(z,s) \leq \varepsilon_2/(4 C_0)$.
    Next, we apply Corollary \ref{corollary1} which says that there exists $s' \in (s/2,s)$ such that
    $$\omega_2(z,s') \leq C_0\big(\omega_p(z,s) + \beta(z,s)\big) \leq \frac{\varepsilon_2}{2}.$$
    The ball $B(z,s')$ satisfies all the required properties because $\beta(z,s') \leq 2\beta(z,s) \leq \varepsilon_2/2$ so that
    $$\omega_2(z,s') + \beta(z,s') \leq \varepsilon_2,$$
    as required.

    It remains to see that $K$ still separates the ball $B(z,s')$.
    But since we know that $\beta(y,t) \leq \varepsilon_0$, that $K$ separates $B(y,t)$ and that $z \in B(y,t/2)$, $s \in (bt, t/2)$, it follows from Lemma \ref{topological1} that $\varepsilon_0$ can be chosen small enough so that $K$ separates $B(z,s')$.
\end{proof}


We are now ready to state one of our main results about the the Hausdorff dimension of the singular set.

\begin{corollary}
    Let $(u,K) \in \mathcal A(\Omega)$ be a minimizer of the Griffith functional with $K$ connected.
    Then there exists a closed set $\Sigma \subset K$ such that $\dim_\mathcal{H}(\Sigma) < 1$ and $K \setminus \Sigma$ is locally a $C^{1,\alpha}$ curve.
\end{corollary}

\begin{proof}
    The proof is standard now that Proposition~\ref{porosity} is established.
    Indeed, we can argue exactly as Rigot in \cite[Remark 3.29]{rigot} which we refer to for more detail.
\end{proof}

\section{Higher integrability of \texorpdfstring{$e(u)$}{e(u)}}

\begin{theorem}\label{thm_integrability}
    Let $(u,K) \in \mathcal A(\Omega)$ be a minimizer of the Griffith functional with $K$ connected and $\HH^1(K) > 0$.
    There exists $C \geq 1$ and $p > 1$ (depending on $\mathbf{A}$) such that the following property holds true. For all $x \in \Omega$, for all $r > 0$ such that $B(x,r) \subset \Omega$,
    \begin{equation*}
        \int_{B\left(x,r/2\right)} \! \abs{e(u)}^{2p} \dd{x} \leq C r^{2-p}.
    \end{equation*}
\end{theorem}

We rely on a higher integrability lemma (\cite[Lemma 4.2]{LM}) which is inspired by the technique of~\cite{DPF}. 
We recall that given $0 < \alpha < 1$, a closed set $K$, $x_0 \in K$ and $r > 0$, we say that $K$ is $C^{1,\alpha}$-regular in the ball $B(x_0,r)$ if it is the graph of a $C^{1,\alpha}$ function $f$ such that, in a convenient coordinate system it holds $f(0) = x_0$, $f'(0) = 0$ and $r^\alpha \| f' \|_{C^\alpha} \leq 1/16$. We take the convention that the $C^{1,\alpha}$ norm is small enough because we don't want it to interfere with the boundary gradient estimates for the Lamé's equations. It is also required by the covering lemma \cite[Lemma 4.3]{LM} on which \cite[Lemma 4.2]{LM} is based.

\begin{lemma}\label{lem_technical}
    We fix a radius $R > 0$.
    Let $K$ be a closed subset of $B(0,R) \subset \R^2$ and $v \colon B(0,R) \to \R^+$ be a non-negative Borel function.
    We assume that there exists $C_0 \geq 1$ and $0 < \alpha \leq 1$ such that the following holds true.
    \begin{enumerate}[label = (\roman*)]
        \item For each ball $B(x,r) \subset B(0,R)$ with $x \in K$,
            \begin{equation*}
                C_0^{-1} r \leq \HH^{1}(K \cap B(x,r)) \leq C_0 r.
            \end{equation*}

        \item For each ball $B(x,r) \subset B(0,R)$ with $x \in K$, there exists a smaller ball $B(y,C_0^{-1} r) \subset B(x,r)$ with $y \in K$ in which $K$ is $C^{1,\alpha}$-regular.

        \item For each ball $B(x,r) \subset B(0,R)$ such that either $K$ is disjoint from $B(x,r)$ or such that $x \in K$ and $K$ is $C^{1,\alpha}$-regular in $B(x,r)$, we have
            \begin{equation*}
                \sup_{B(x,r/2)} v(x) \leq C_0 \left(\frac{R}{r}\right).
            \end{equation*}
    \end{enumerate}
    Then there exists $p > 1$ and $C \geq 1$ (depending on $C_0$) such that
    \begin{equation*}
        \fint_{B(0,R/2)} \! v^p \leq C.
    \end{equation*}
\end{lemma}

\begin{proof}[Proof of Theorem~\ref{thm_integrability}]
    We apply the Lemma \ref{lem_technical}.
    More precisely, for all $x \in \Omega$ and all $R > 0$ such that $B(x,R) \subset \Omega$, one can applies Lemma~\cite[Lemma 4.2]{LM} in the ball $B(x,R)$ to the function $v := R \abs{e(u)}^2$.
    The assumption (i) follows from the local Ahlfors-regularity of $K$.
    The assumption (ii) follows from the porosity (Proposition~\ref{porosity}).
    The assumption (iii) follows from interior/boundary gradient estimates for the Lame's equations and from the local Ahlfors-regularity.
    In particular, the boundary estimate is detailed in Lemma~\ref{lem_lame_estimate} in Appendix~\ref{appendix_lame}.
\end{proof}


\begin{appendices}

    \section{Lamé's equations}\label{appendix_lame}

    We work in the Euclidean space $\R^n$ ($n \geq 2$).
    For $r > 0$, $B_r$ denotes the ball of radius $r$ and centered at $0$.
    We fix a radius $0 < R \leq 1$, an exponent $0 < \alpha \leq 1$, a constant $A > 0$ and a $C^{1,\alpha}$ function $f\colon \R^{n-1} \cap B_R \to \R$ such that $f(0) = 0$, $\nabla f(0) = 0$ and $R^\alpha \left[\nabla f\right]_\alpha \leq A$.
    We introduce
    \begin{align*}
        V_R      & := \set{x \in B_R | x_n > f(x')} \\
        \Gamma_R & := \set{x \in B_R | x_n = f(x')}.
    \end{align*}
    We denote by $\nu$ the normal vector field to $\Gamma_R$ going upward.
    For $0 < t \leq 1$, we write $t V_R$ for $V_R \cap B_t$ and $t \Gamma_R$ for $\Gamma_R \cap B_t$.
    For $u \in W^{1,2}(V_R;\R^n)$, we denote by $u^*$ the trace of $u$ in $L^2(\partial V_R;\R^n)$.
    For a function $\xi \colon V_R \to \R^{n \times n}$, we define (formally) $\mathrm{div}(\xi)$ as the vector field whose $i$th coordinate is given by $\mathrm{div}(\xi)_i = \sum_j \partial_j \xi_{ij}$.
    We also recall the notation for the linear strain tensor
    \begin{equation*}
        e(u) = \frac{Du + Du^T}{2}.
    \end{equation*}
    and the stress tensor
    \begin{equation*}
        \mathbf{A}e(u) = \lambda \mathrm{div}(u) I_n + 2 \mu e(u),
    \end{equation*}
    where $\lambda$ and $\mu$ are the Lamé coefficients satisfying $\mu > 0$ and $\lambda + \mu > 0$.
    We denote by $W^{1,2}_0(V_R \cup \Gamma_R;\R^n)$ the space of functions $v \in W^{1,2}(V_R;\R^n)$ such that $v^* = 0$ on $\partial V_R \setminus \Gamma_R$.

    Our object of study are the functions $u \in W^{1,2}(V_R) \cap L^\infty(V_R)$ which are weak solutions of
    \begin{equation}\label{eq_lame_problem}
        \left\{
            \begin{array}{lcl}
                \mathrm{div}(\mathbf{A}e(u)) & = & 0 \quad \text{in $V_R$} \\
                \mathbf{A}e(u) \cdot \nu     & = & 0 \quad \text{on $\Gamma_R$},
            \end{array}
        \right.
    \end{equation}
    that is for all $v \in W^{1,2}_0(V_R \cup \Gamma_R;\R^n)$,
    \begin{equation*}
        \int_{V_R} \! \mathbf{A}e(u) : Dv \dd{x} = 0.
    \end{equation*}

    \begin{remark}
        As
        \begin{equation*}
            \mathbf{A}e(u) = (\lambda + \mu) Du^T + \mu \! Du + \lambda (\mathrm{div}(u) I_n - Du^T)
        \end{equation*}
        and the part $\mathrm{div}(u) I_n -  Du^T$ is divergence free, we can also write formally
        \begin{equation*}
            \mathrm{div}(\mathbf{A}e(u)) = (\lambda + \mu) \nabla \mathrm{div}(Du) + \mu \Delta u.
        \end{equation*}
    \end{remark}

    We are going to justify the following estimate.

    \begin{lemma}\label{lem_lame_estimate}
        Let us assume $n=2$.
        There exists $C \geq 1$ (depending on $\alpha$, $A$, $\lambda$, $\mu$) such that
        \begin{equation*}
            \sup_{\tfrac{1}{2} V_R} \abs{e(u)} \leq C \left(\fint_{V_R} \! \abs{e(u)}^2 \dd{x}\right)^{\frac{1}{2}}.
        \end{equation*}
    \end{lemma}

    \begin{proof}
        It suffices to prove that for all solution of \ref{eq_lame_problem}, we have
        \begin{equation}\label{eq_lame_estimate0}
            \sup_{\tfrac{1}{2}V_R} \abs{Du} \leq C \left(\fint_{V_R} \! \abs{Du}^2 \dd{x}\right)^{\frac{1}{2}}.
        \end{equation}
        Indeed, we observe first that $\abs{e(u)} \leq \abs{Du}$ so (\ref{eq_lame_estimate0}) implies
        \begin{equation}\label{eq_lame_estimate1}
            \sup_{\tfrac{1}{2} V_R} \abs{e(u)} \leq C \left(\fint_{V_R} \! \abs{Du}^2 \dd{x}\right)^{\frac{1}{2}}.
        \end{equation}
        By Korn inequality, there exists a skew-symmetric matrix $R$ such that
        \begin{equation*}
            \int_{V_R} \! \abs{Du - R}^2 \dd{x} \leq \int_{V_R} \! \abs{e(u)}^2 \dd{x}
        \end{equation*}
        so it is left to apply (\ref{eq_lame_estimate1}) to $x \mapsto u(x) - Rx$, which also solves Lamé's equations.

        From now on, we deal with (\ref{eq_lame_estimate0}).
        The letter $C$ plays the role of a constant $\geq 1$ that depends on $\lambda$, $\mu$ and $\alpha$, $A$.
        We refer to the proof \cite[Theorem~3.18]{FFLM} which itself refers to the proof of \cite[Theorem~7.53]{Ambrosio_BV}.
        We straighten the boundary $\Gamma_R$ via the $C^{1,\alpha}$ diffeomorphism $\phi \colon x \mapsto x' + (x_n - f(x')) e_n$.
        We observe that $\phi(\overline{V_R})$ contains a half-ball ball $B^+ = \overline{B}(0,C_0^{-1} R)^+$, where $C_0 \geq 1$ is a constant that depends on $\lambda$, $\mu$, $\alpha$.
        The Neumann problem satisfied by $u$ in $V_R$ is transformed into a Neumann problem satisfied by a function $v$ in $\overline{B}(0, C_0^{-1} R)^+$.
        Then we symmetrize the elliptic system to the whole ball $B = B(0,C_0^{-1} R)$ as in \cite[Theorem~3.18]{FFLM}.
        Following the proof of \cite[Theorem~7.53]{Ambrosio_BV} (in the special case where the right-hand side $h$ is zero), we arrive to the fact there exists $q > n = 2$ (depending on $\lambda$, $\mu$, $\alpha$) such that for all $x_0 \in \tfrac{1}{2} B$ and $0 < \rho \leq r \leq C^{-1} R$
        \begin{multline*}
            \int_{B_\rho(x_0)} \! \abs{\nabla v - (\nabla v)_{x_0,\rho}}^2 \dd{x} \leq C \left(\frac{\rho}{r}\right)^q \int_{B_r(x_0)} \! \abs{\nabla v - (\nabla v)_{x_0,r}}^2 \dd{x} \\ + C \rho^q \int_B \! \abs{\nabla v}^2 \dd{x}.
        \end{multline*}
        In particular by Poincaré-Sobolev inequality,
        \begin{equation*}
            \int_{B_\rho(x_0)} \! \abs{\nabla v - (\nabla v)_{x_0,\rho}}^2 \dd{x} \leq C \left(\frac{\rho}{r}\right)^q \int_B \! \abs{\nabla v}^2 \dd{x}.
        \end{equation*}
        According to the Campanato characterisation of Hölder spaces,
        \begin{equation*}
            \left[\nabla v\right]_{C^{0,\sigma}(\tfrac{1}{2}B)} \leq C \left(R^{-(2+2\sigma)} \int_B \! \abs{\nabla v}^2 \dd{x}\right)^\frac{1}{2},
        \end{equation*}
        where $\sigma = \frac{q - 2}{2}$, and this implies
        \begin{equation*}
            \sup_{\tfrac{1}{2}B} \abs{\nabla v} \leq C \left(\fint_B \! \abs{\nabla v}^2 \dd{x}\right)^2.
        \end{equation*}
        This property is inherited by $u$ via the diffeomorphism $\phi$,
        \begin{equation*}
            \sup_{C^{-1} V_R} \abs{\nabla u} \leq C \left(\fint_{V_R} \! \abs{\nabla u}^2 \dd{x}\right)^2.
        \end{equation*}
        We can finally bound the supremum of $\abs{\nabla u}$ on $\tfrac{1}{2}V_R$ by a covering argument.
    \end{proof}

\end{appendices}

\section*{Acknowledgements}

This work was co-funded by the European Regional Development Fund and
the Republic of Cyprus through the Research and Innovation Foundation
(Project: EXCELLENCE/1216/0025). The authors both have support from the
CNR-GDR CALVA founds and ANR SHAPO. A part of this work was done while Camille
Labourie visited Antoine Lemenant at the university of Lorraine (Nancy,
France) in march 2021. This was just the day before the french third
confinement time.

\bibliographystyle{plain}
\bibliography{biblio_griffith}
\end{document}